\documentclass[a4paper,10pt,twoside]{article}

\usepackage{algorithm}
\usepackage{algorithmic}
\usepackage{amsmath}
\usepackage{amsopn}
\usepackage{amssymb}
\usepackage{amsthm}
\usepackage{authblk}
\usepackage[british]{babel}
\usepackage{caption}
\usepackage{subcaption}
\usepackage{color}
\usepackage{fancyhdr}
\usepackage{float}
\usepackage[T1]{fontenc}
\usepackage{graphicx}
\usepackage{hyperref}
\usepackage{nomencl}
\usepackage{makeidx}
\usepackage{mathrsfs} 
\usepackage{moreverb}
\usepackage{pdfpages} 
\usepackage{tikz} 
\usepackage{rotating}
\usepackage[nottoc, notlof, notlot]{tocbibind} 
\usepackage{url}
\usepackage{wrapfig} 
\usepackage[all]{xy}

\newtheorem{de}{Definition}
\newtheorem{ex}{Example}
\newtheorem{lem}{Lemma}
\newtheorem{prop}{Proposition}
\newtheorem{rem}{Remark}
\newtheorem{theo}{Theorem}
\newtheorem{me}{ faire}
\newtheorem{pt}{Point}

\def\ie{\textit{i.e.}}

\DeclareMathOperator{\CC}{\mathbb{C}}

\DeclareMathOperator{\FF}{\mathbb{F}}

\DeclareMathOperator{\Frac}{Frac}

\DeclareMathOperator{\GL}{GL}

\DeclareMathOperator{\PP}{\mathbb{P}}

\DeclareMathOperator{\SL}{SL}

\DeclareMathOperator{\ZZ}{\mathbb{Z}}

\newcommand{\fullref}[1]{%
  \ref{#1} 
  }

\theoremstyle{plain}

\theoremstyle{definition}

\theoremstyle{remark}

\numberwithin{equation}{section}

\def\phi{\varphi}

\definecolor{liensHypertextes}{rgb}{0,0,0.5}
\definecolor{liensBiblio}{rgb}{0,0.5,0}
\hypersetup{
pageanchor=true,
pdfpagelabels = true,
hyperindex = true,
hyperfigures = true,
colorlinks=true,
linkcolor=liensHypertextes,
citecolor=liensBiblio,
pdfauthor={Florent Ulpat Rovetta}
}
\title{A strategy and a new operator to generate covariants in small characteristic}
\author{Florent Ulpat Rovetta}

\begin{document}
\maketitle

\begin{abstract}
\noindent
We present some new results about covariants in small characteristic. 
In section \ref{Sturmfels}, we give a method to construct covariants using an approach similar to Sturmfels. 
We apply our method to find a separating system of covariants for binary quartic in characteristic $3$. 
In the second section, we construct a new operator on covariants when the characteristic is small compared to the degree of the form.
\end{abstract}

\section*{Introduction}
We are interested in the computation of covariants of binary forms in small characteristic in a similar optical of \cite{Basson} and in association with the moduli space of hyperelliptic curves.
Although, in characteristic $0$ or in large characteristic, it is a classic problem (but still formidable in practice when the degree of the form is higher than $10$), in small characteristic, the approach based on transvections (whose are differentials operators) collapse.
In this context, we wanted to test the effectiveness of an alternative method following  \cite{MR0374142} and \cite{MR2667486}. 
The idea is to consider the algebra of covariants of $n$-points of $\PP^1$ under the action of $\GL_2(k)$ (cf. definition \fullref{deCov_nPoints}).
The main advantage of this algebra is that it admits a generator system of covariants explicit and \emph{independent} of the characteristic. 
The covariants for the binary forms are then obtained as the sub-algebra symmetrised by ${\cal S}_n$. 
Although attractive in theory, our current implementation is extremely limited. Indeed the group action ${\cal S}_n$, in the modular case (i.e when the characteristic divide the order of the group), under the covariants of $n$ points fails with generics algorithms of {\tt Magma} as soon as  $n=6$ (cf. \cite[sec.5.2.6]{UlpatRovetta}). 
However, this method was used to determine a separating set (a lower condition than being a generator system, see the definition \fullref{defInvSep}) in the characteristic $3$ binary quartics case. 
Along the way, we realized that some of the invariants/covariants appearing in small characteristic could be derived from classical covariants by a new easy differential operation (cf.  section \ref{subsubsecCovDerMultiple}) under certain conditions that we clarify. 
For the octavics, we get the new invariant of degree $1$ found by  \cite{Basson} and new covariants in degree $4$ and $6$  (cf.  page \pageref{newcov}). 
This operation, while it enriches the algebra of covariants obtained by reduction of those in characteristic zero, is not sufficient to get all the covariants (as we will see in an example in degree  $4$ at the end of this paper).
The question of efficient generation in small characteristic remains wide open.\\

\noindent
{\bf Notations.} 
Let $p$ be a prime number or $0$ and let $k$ be an algebraic closed field of characteristic $p$ and ${\cal C}_n$ the algebra of binary covariants defined over $k$.

\section{A strategy to construct covariants in small characteristic}
\label{Sturmfels}
Except for quartics (cf. \cite[sec 2.10.2]{Basson}) and Igusa invariants for sextic, we do not know a generator system of invariants in every characteristic. 
Thanks to clever reductions and many computations, Basson exhibited in his thesis a ``separating'' system\footnote{\ie: separating the orbits (cf definition \fullref{defInvSep})}. 
He conjectures that it is generator in characteristic $3$ and $7$ for the octavic. 
For the characteristic $p\geq 11$, generator systems are known thanks to the results of \cite{MR2990029}.
To get new results for covariants, we will establish a totally different computation method following\cite{MR0374142} and \cite{MR2667486}. 
We obtain new results for covariants of binary quartic in characteristic $3$.

\subsection{Strategy}
\label{subsec_resultSturmfels}
The study of covariants of $n$-points of $\PP^1(k)$ under the action of $\GL_2(k)$ is a classical framework and we recall here the principal results. 
The main advantage of this work is that there exists an explicit generator system of covariants \emph{independent} of the characteristic. 
Then, the covariants for binary forms come from the sub-algebra symmetrised by ${\cal S}_n$.\\

\noindent
We slightly modify the results of Sturmfels \cite[Chap 3, sec 6]{MR2667486} in order to be valid in every characteristic. 
In the case of invariants, this is exactly Geyer's method \cite{MR0374142}. 
Let $n>1$ be a positive integer. 
Consider the binary form:
\begin{align*}
f(x,z) & = \sum_{k=0}^na_kx^kz^{n-k}\\
         & = (\mu_1x-\nu_1z)(\mu_2x-\nu_2z)\dots(\mu_nx-\nu_nz).
\end{align*}
The `roots' $(\mu_i,\nu_i)$ can be seen as points $(\mu_i,\nu_i)\in\PP^1$.
\begin{de}
Let M be a monomial in ${k}[\mu_1, \nu_1, \mu_2, \nu_2,\dots,\mu_n, \nu_n, \:x,\:z]$ such that:
$$M=\mu_1^{u_1}\mu_2^{u_2}\cdots\mu_n^{u_n}\nu_1^{v_1}\nu_2^{v_2}\cdots\nu_n^{v_n}x^{w_1}z^{w_2}$$
and $P$ be a polynomial in ${k}[\mu_1, \nu_1, \mu_2, \nu_2,\dots,\mu_n, \nu_n, \:x,\:z]$.
We say that:
\begin{itemize}
\item $M$ is \emph{regular of degree} $d$ if $u_1+v_1 = u_2+v_2 = \dots = u_n+v_n = d.$
The integer $d$ is called the \emph{regularity degree} of $M$.
\item $P$ is \emph{regular of degree} $d$ if all of its monomials are regular of degree $d$. 
When $P$ is regular for a degree $d$, we say that $P$ is \emph{regular}. 
\item $P$ is \emph{symmetric} if, for all permutation $\sigma\in S_n$:
$$P(\mu_1, \nu_1, \mu_2, \nu_2,\dots,\mu_n, \nu_n, \:x,\:z) = P(\mu_{\sigma(1)}, \nu_{\sigma(1)}, \mu_{\sigma(2)}, \nu_{\sigma(2)},\dots,\mu_{\sigma(n)}, \nu_{\sigma(n)}, \:x,\:z).$$
\end{itemize}
A regular monomial is \emph{reducible} if it can be expressed as the product of two regular monomials of regularity degree great than equal to $1$.
\end{de}
\noindent
We define the action of $\GL_2({k})$ on ${k}[\mu_1, \nu_1, \mu_2, \nu_2,\dots,\mu_n, \nu_n, \:x,\:z]$ in the following way:
Let $M\in\GL_2({k})$, 
$$
\left(
\begin{array}{c}
\nu_i\\
\mu_i
\end{array}
\right)
\to
\left(
\begin{array}{c}
\overline{\nu}_i\\
\overline{\mu}_i
\end{array}
\right)
= M^{-1}\cdot
\left(
\begin{array}{c}
\nu_i\\
\mu_i
\end{array}
\right)
$$
$$
\left(
\begin{array}{c}
x\\
z
\end{array}
\right)
\to
\left(
\begin{array}{c}
\overline{x}_i\\
\overline{z}_i
\end{array}
\right)
= M^{-1}\cdot
\left(
\begin{array}{c}
x\\
z
\end{array}
\right).
$$
\label{deCov_nPoints}
A regular polynomial $P$ is a \emph{covariant (of $n$ points)}\index{covariant! of $n$ points} if there exists $w\in \ZZ$ such that:
$$P(\overline{\mu}_1, \overline{\nu}_1, \overline{\mu}_2, \overline{\nu}_2,\dots,\overline{\mu}_n, \overline{\nu}_n, \:\overline{x},\:\overline{z}) = \det(M)^wP(\mu_1, \nu_1, \mu_2, \nu_2,\dots,\mu_n, \nu_n, \:x,\:z).$$
$P$ is called an \emph{invariant (of $n$ points)} if it does not depend on $x$ and $z$.
It is easy to define covariants quantities thanks to the brackets. 
Let $1\leq i<j\leq n$, we call \emph{bracket} the following quantities:
\begin{align*}
[ij] &:= \mu_i\nu_j-\nu_i\mu_j,\\
[iu] &:= \mu_ix-\nu_iz.
\end{align*}
\label{deBreg}
The sub-ring ${\cal B}(n)$\nomenclature{${\cal B}(n)$}{The sub-ring of ${k}[\mu_1, \nu_1, \mu_2, \nu_2,\dots,\mu_n, \nu_n, \:x,\:z]$ generated by the brackets, page \pageref{deBreg}} generated by these brackets in  ${k}[\mu_1, \nu_1, \mu_2, \nu_2,\dots,\mu_n, \nu_n, \:x,\:z]$, is called \emph{the bracket ring}\index{the bracket ring}. 
We also denote ${\cal B}_{reg}(n)$ the sub-ring of ${\cal B}(n)$ of polynomial in the brackets which are regular of degree $d$ for $d\geq 0$\nomenclature{${\cal B}_{reg}(n)$}{The sub-ring of ${\cal B}(n)$ of polynomial in the brackets how are regular of degree $0$, page \pageref{deBreg}}. 
The latter is generated by the monomials of the form:
$$\prod_{i<j}[ij]^{m_{ij}},$$
where the integers $m_{ij}$ verify $d = \sum_{j=1}^{i-1} m_{ji} + \sum_{j=i+1}^nm_{ij}$. 
\label{theoBregCovReguliers}
The polynomial ring of regular covariants is equal to ${\cal B}_{reg}(n)$ (consequence of the first fundamental theorem (cf \cite{MR0000255})). 
When the acting group is $\GL_2({\CC})$, the theorem 3.2.1 and the lemma 3.6.5 of \cite{MR2667486} provide a demonstration. 
When the group is arbitrary, the proof is in \cite{MR0422314}.
Note also that \cite[Satz 5]{MR0374142} gives an elementary proof in the case of $\GL_2(k)$.
\noindent
In the section \ref{subsecBreg_n}, we present an example of computation of generators of ${\cal B}_{reg}(n)$. 
What remains to describe the final stage to get the of binary forms
Let:
\begin{align*}
\Psi:{k}[a_0,a_1,\dots,a_n,\:x,\:z] & \to{k}[\mu_1, \nu_1, \mu_2, \nu_2,\dots,\mu_n, \nu_n, \:x,\:z]\\
a_{n-k} &\to (-1)^n\mu_1\cdots\mu_n\cdot\sigma_k(\frac{\nu_1}{\mu_1},\dots,\frac{\nu_n}{\mu_n}).
\end{align*}
$\sigma_k$ represents the $k$-th elementary symmetric polynomial function in $n$ variables. 
The following theorem (from \cite[th3.6.6]{MR2667486}) is an elementary consequence of the previous theorem.

\noindent
\label{BregSym}
Let ${\cal B}_{reg,sym}(n)$ be the sub-ring of ${\cal B}_{reg}(n)$ of the polynomial in the brackets which are symmetric.\nomenclature{${\cal B}_{reg,sym}(n)$}{The sub-ring of ${\cal B}_{reg}(n)$ of the polynomial in the brackets who are symmetric, page \pageref{BregSym}}
\begin{theo}
\label{theo_covbrakets}
The mapping $\Psi$ is an isomorphism between the ring of covariants of binary forms ${k}[a_0,\dots,a_n,x,z]^{\SL_2(\overline{k})}$ and the sub-ring  ${\cal B}_{reg,sym}(n)$ of symmetric and regular polynomial brackets functions of ${k}[\mu_1, \nu_1, \mu_2, \nu_2,\dots,\mu_n, \nu_n, \:x,\:z]$. 
If $C(a_0,\dots,a_n)$ is a covariant of degree $d$ and order $r$ then $\Psi(C)$ is a symmetric polynomial bracket function such that:
\begin{enumerate}
\item in every monomial of $\Psi(C)$,the index  $1$, $2$,\dots, $n$ appears $d$ times,
\item  in every monomial of $\Psi(C)$, the letter $u$ appears $r$ times.
\end{enumerate}
\end{theo}

\subsection{Computation of ${\cal B}_{reg}(n)$}
\label{subsecBreg_n}
We wish to compute ${\cal C}_n$ as ${\cal B}_{reg}(n)^{{\cal S}_n}$. 
If $b_1,\dots,b_t$ is a generator system of bracket monomial for ${\cal B}_{reg}(n)$, then we have a surjective morphism:
\[
\begin{array}{cll}
k[x_1,\dots, x_t] & \to & {\cal B}_{reg}(n).\\
x_i & \to & b_i
\end{array}
\]
The kernel $I$ of this morphism is generated by the following relations.
\begin{prop}
\label{prop_sysigie}
Let $1\leq i<j<k<l\leq n$, we have:
$$[ik][jl] = [ij][kl] + [il][jk],$$
$$[ik][ju] = [ij][ku] + [iu][jk].$$
These relations are called the \emph{syzygies}.
\end{prop}
\noindent
In order to have a clearer view of a generator system of ${\cal B}_{reg}(n)$ the monomials will be represented by weighted graphs such that the vertices form a regular polygon. 
We represent:
\begin{itemize}
\item a monomial of ${\cal B}(n)$ by a graph with $n$ vertices numbered from $1$ to $n$ and a vertex called $\textcircled{u}$,
\item the bracket $[ij]$ by an edge connecting the vertex $i$ to the vertex $j$ with $i<j\in\{1,\dots,n\}$,
\item the bracket $[iu]$ by an edge connecting the vertex $\textcircled{u}$ to the vertex $i$.
\end{itemize}
For example, the bracket product $[12][14]^3[34][1u][2u]^2\in{\cal B}(4)$
is represented by the following weighted graph:
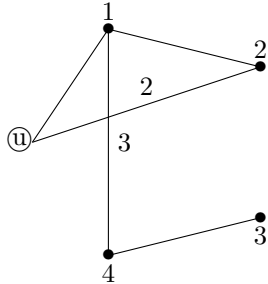
\begin{center}
\begin{tikzpicture}
\draw (0.13,0) node [left] {$\textcircled{u}$};
\draw (1,1.5) node {$\bullet$} node [above]{$1$}; 
\draw (3,1) node {$\bullet$} node [above]{$2$};
\draw (3,-1) node {$\bullet$} node [below]{$3$};
\draw (1,-1.5) node {$\bullet$} node [below]{$4$};
\draw(3,1)--(1,1.5);                                                                 
\draw(1,-1.5)--(1,1.5) node [midway,right]{$3$};              
\draw(0,0)--(1,1.5);                                                                 
\draw(0,0)--(3,1) node [midway,above]{$2$};                   
\draw(3,-1)--(1,-1.5);                                                               
\end{tikzpicture}
\end{center}

\noindent
The previous allow us to formulate $5$ remarks which are very useful to construct a generator system for ${\cal B}_{reg}(n)$. 
The next point follows from the definition of ${\cal B}_{reg}(n)$.
\begin{pt}
\label{pt_reg}
Every monomial of order $m$ and of regularity degree $d$ is represented by a graph with $m$ connexions with $\textcircled{u}$ and every numbered vertex has a valence $d$. 
\end{pt}
\noindent
Moreover, by Kempe's lemma (\cite[th. 3.7.3 p. 132]{MR2667486}), the covariant algebra of $n$ points is generated by elements of regularity degree at most $2$. 
Also, by \cite[th.2.3 p.7]{MR2477759}, the invariants of $n$ points are generated by the regularity degree $1$, hence the following point:
\begin{pt}
\label{pt_bornes_degre}
The numbered vertices have a valence at most $2$. 
The graph corresponding to invariants has a valence $1$.
\end{pt}
\noindent
By its definition, if a graph is expressed as a union of subgraphs corresponding to graphs of smaller degree and smaller order already computed, the associated covariant is reducible.
\begin{pt}
\label{pt_sous_graphe}
The graphs having a subgraph already calculated are excluded.
\end{pt}
\noindent
When the vertices are in a regular polygon, proposition\fullref{prop_sysigie} causes the following point (see \cite[th.6.2 p.72]{MR722856} for a proof):
\begin{pt}
\label{pt_syzygies}
The graphs of our generator system have no edges crossing.
\end{pt}
\noindent
Thanks to point \ref{pt_bornes_degre},  the number of adjacent edges of $\textcircled{u}$ is bounded. 
Moreover, when $n$ is even, point \ref{pt_reg} imposes another condition on the edges adjacent to $\textcircled{u}$.
\begin{pt}
\label{pt_bormnes_ordre}
$\textcircled{u}$ has at most $2n$ adjacent edges.\\
When $n$ is even, $\textcircled{u}$ has an even number of adjacent edges.
\end{pt}
\begin{ex}
\label{exBreg_quartic}
Using the five points above and considering increasing orders we get the following generators, when $n=4$, called $t_0$, $t_1$, $u_0$, $u_1$, $u_2$ and $f$.
\end{ex}

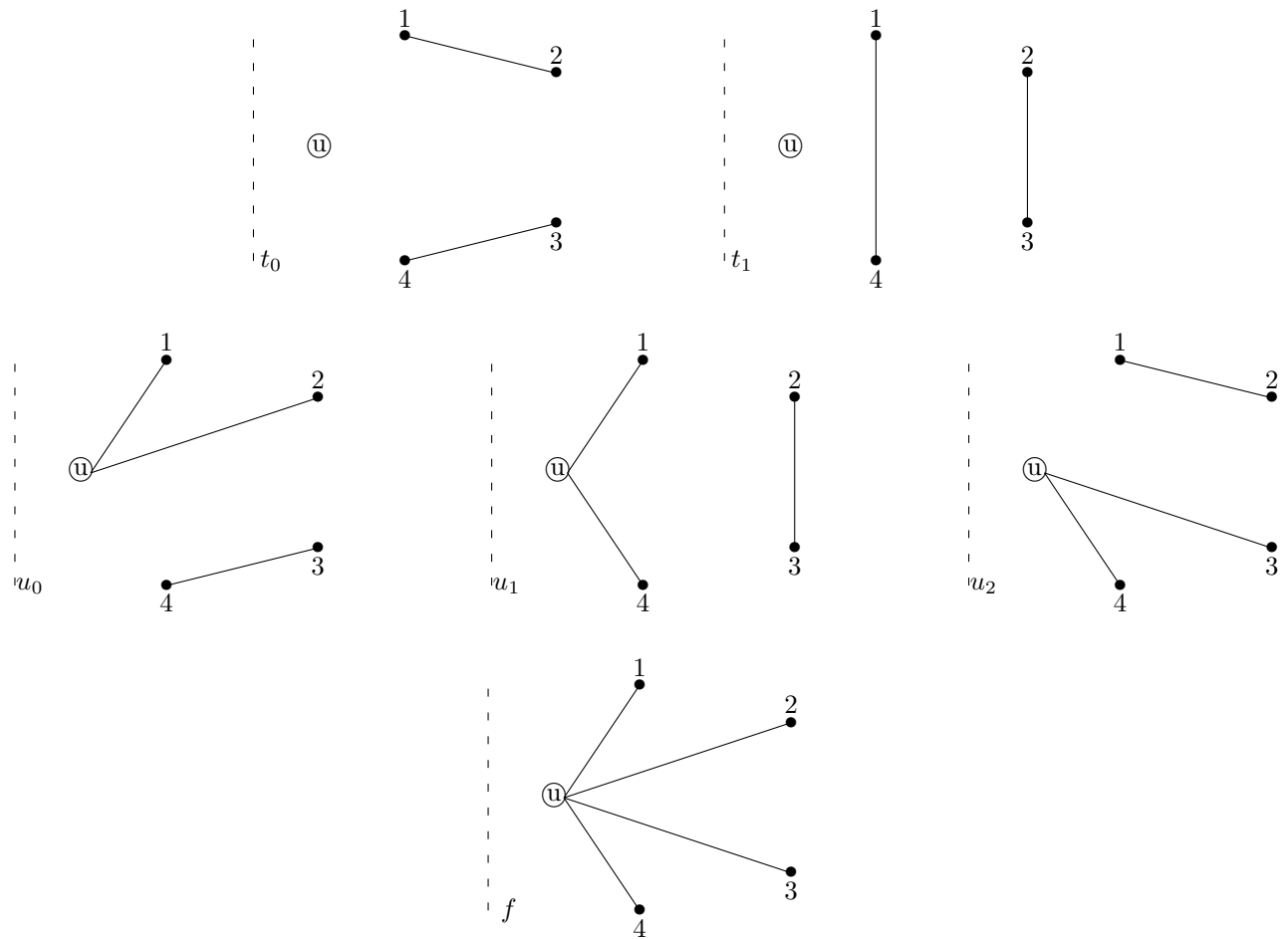
\begin{center}
\begin{tikzpicture}
\draw [loosely dashed] (-1,-1.5) -- (-1,1.5);
\draw (-0.5,-1.5) node [left] {$t_0$};
\draw (0.17,0) node [left] {$\textcircled{u}$};                    
\draw (1,1.5) node {$\bullet$} node [above]{$1$}; 
\draw (3,1) node {$\bullet$} node [above]{$2$};
\draw (3,-1) node {$\bullet$} node [below]{$3$};
\draw (1,-1.5) node {$\bullet$} node [below]{$4$};
\draw(3,1)--(1,1.5);                                                             
\draw(3,-1)--(1,-1.5);                                                          
\end{tikzpicture}\:\:\:\:\:\:\:\:\:\:\:\:\:\:\:\:\:\:\:\:\:\:\:\:
\begin{tikzpicture}
\draw [loosely dashed] (-1,-1.5) -- (-1,1.5);
\draw (-0.5,-1.5) node [left] {$t_1$};
\draw (0.17,0) node [left] {$\textcircled{u}$};                    
\draw (1,1.5) node {$\bullet$} node [above]{$1$}; 
\draw (3,1) node {$\bullet$} node [above]{$2$};
\draw (3,-1) node {$\bullet$} node [below]{$3$};
\draw (1,-1.5) node {$\bullet$} node [below]{$4$};
\draw(1,-1.5)--(1,1.5);                                                      
\draw(3,1)--(3,-1);                                                            
\end{tikzpicture}
\end{center}
\begin{center}
\begin{tikzpicture}
\draw [loosely dashed] (-1,-1.5) -- (-1,1.5);
\draw (-0.5,-1.5) node [left] {$u_0$};
\draw (0.17,0) node [left] {$\textcircled{u}$};                    
\draw (1,1.5) node {$\bullet$} node [above]{$1$}; 
\draw (3,1) node {$\bullet$} node [above]{$2$};
\draw (3,-1) node {$\bullet$} node [below]{$3$};
\draw (1,-1.5) node {$\bullet$} node [below]{$4$};
\draw(0,0)--(1,1.5);                                                                
\draw(0,0)--(3,1);                                                                   
\draw(3,-1)--(1,-1.5);                                                             
\end{tikzpicture}\:\:\:\:\:\:\:\:\:\:\:\:\:\:\:\:\:\:\:\:\:\:\:\:
\begin{tikzpicture}
\draw [loosely dashed] (-1,-1.5) -- (-1,1.5);
\draw (-0.5,-1.5) node [left] {$u_1$};
\draw (0.17,0) node [left] {$\textcircled{u}$};                    
\draw (1,1.5) node {$\bullet$} node [above]{$1$}; 
\draw (3,1) node {$\bullet$} node [above]{$2$};
\draw (3,-1) node {$\bullet$} node [below]{$3$};
\draw (1,-1.5) node {$\bullet$} node [below]{$4$};
\draw(0,0)--(1,1.5);                                                                
\draw(3,1)--(3,-1);                                                                  
\draw(0,0)--(1,-1.5);                                                               
\end{tikzpicture}\:\:\:\:\:\:\:\:\:\:\:\:\:\:\:\:\:\:\:\:\:\:\:\:
\begin{tikzpicture}
\draw [loosely dashed] (-1,-1.5) -- (-1,1.5);
\draw (-0.5,-1.5) node [left] {$u_2$};
\draw (0.17,0) node [left] {$\textcircled{u}$};                    
\draw (1,1.5) node {$\bullet$} node [above]{$1$}; 
\draw (3,1) node {$\bullet$} node [above]{$2$};
\draw (3,-1) node {$\bullet$} node [below]{$3$};
\draw (1,-1.5) node {$\bullet$} node [below]{$4$};
\draw(3,1)--(1,1.5);                                                                
\draw(0,0)--(3,-1);                                                                  
\draw(0,0)--(1,-1.5);                                                               
\end{tikzpicture}
\end{center}
\begin{center}
\begin{tikzpicture}
\draw [loosely dashed] (-1,-1.5) -- (-1,1.5);
\draw (-0.5,-1.5) node [left] {$f$};
\draw (0.17,0) node [left] {$\textcircled{u}$};
\draw (1,1.5) node {$\bullet$} node [above]{$1$}; 
\draw (3,1) node {$\bullet$} node [above]{$2$};
\draw (3,-1) node {$\bullet$} node [below]{$3$};
\draw (1,-1.5) node {$\bullet$} node [below]{$4$};
\draw(0,0)--(1,1.5);                                                                
\draw(0,0)--(3,1);                                                                   
\draw(0,0)--(3,-1);                                                                  
\draw(0,0)--(1,-1.5);                                                               
\end{tikzpicture}
\end{center}

\subsection{Symmetrization}
The action of ${\cal S}_n$ on the $b_i$ induces a representation $G_n$ of ${\cal S}_n$ in $\GL_t(k)$. 
There are algorithms for computing $k[x_1,\dots,x_t]^{G_n} = R_n$ (cf \cite{DerKem}). 
The latter being also valid in the modular case (i.e when $p\:\big|\:|G_n|$), we use them `naively' through their {\tt Magma} implementations. 
However this process generates a limitation when $n=6$.
What remains to clarify is the link between $R_n$ and ${\cal C}_n$.\\

\noindent
When $p$ does not divide $|{\cal S}_n|$, ${\cal S}_n$ is a linearly reductive group (cf \cite[Def 2.2.1]{DerKem}) and the existence of Reynolds operators (cf \cite[Th 2.2.5]{DerKem}) preserves the surjectivity of the morphism $k[x_1,\dots,x_n]\to{\cal B}_{reg}(n)$ in the symmetrization process. 
Thanks to \cite[lem.3.7.2]{MR2667486}, the image of a generator system of $R_n$ by the canonical surjection $k[x_1,\dots,x_n]\to k[x_1,\dots,x_n]/I = {\cal B}_{reg}(n)$ is a generator system of ${\cal C}_n = {\cal B}_{reg}(n)^{{\cal S}_n}$. 
In particular, if $p>n$, we get a generalization of the result of Geyer:
the covariant ring ${\cal C}_n$ is the reduction modulo $p$ of the covariant ring in characteristic $0$. 
In particular, the bi-graduate Poincaré series are identical.\\

\noindent
When $p\:\big|\:|{\cal S}_n|$, ${\cal S}_n$ is only a reductive group (cf. \cite[sec 2.2.2]{DerKem}) and the previous result is no longer valid in the general case. 
To overcome this, we recall the following concept:
\begin{de}
\label{defInvSep}
Let $X$ be an affine variety and $G$ an automorphism group of $k[X]$. 
A subset $S\subseteq k[X]^G$ is called \emph{separating}\index{separating} if, for every couple of points $(x,y)$ of $X$, we have the following propriety:
if there exists an element $f\in k[X]^G$ such that $f(x)\neq f(y)$, there exists an element $g$ in $S$ such that $g(x)\neq g(y)$. 
\end{de}
\noindent
The relation with the invariant ring is the following (cf \cite[prop.2.3.10]{DerKem}):
\begin{prop}
Suppose that $X$ is irreducible and $k[X]^G$ is finitely generated. 
Let $A\subseteq k[X]^G$ be a separating sub-algebra finitely generated. 
$\Frac(k[X]^G)$ is then a purely inseparable finite extension of $\Frac(A)$. 
In particular, if the characteristic of $k$ is zero then:
$$\Frac(A) = \Frac(k[X]^G).$$
\end{prop}
\noindent
Definition \ref{defInvSep} has the advantage to preserve the surjectivity on transition to invariants. 
Let $G$ be a linear algebraic group. 
Thanks to \cite[p. 59]{DerKem}, if $G$ is reductive, $G$ regularly acts on an affine variety $X$ and $Y\subseteq X$ is a sub-variety $G$-stable then the restriction map $k[X]\to k[Y]$ sends a separating subset of $k[X]^G$ to a separating subsets of $k[Y]^G$. 
So $R_n$ is the separating algebra of covariants of binary forms of degree $n$ but we do not necessarily have an equality between $R_n$ and ${\cal C}_n$. 
We will see, in the following case of quartics in characteristic $3$, when the inclusion is strict.
\begin{ex}
\label{exCov4}
In example \ref{exBreg_quartic}, we have seen that the covariant algebra of $4$ points is generated by $t_0$, $t_1$, $u_0$, $u_1$, $u_2$ et $f$. 
We will make the group ${\cal S}_4$ act  and, using the function {\tt InvariantRing} of {\tt Magma}, we will compute a separating system of the covariant algebra ${\cal C}_4$. 
Knowing that ${\cal S}_4$ is generated by $\sigma = (1234)$ and $\tau =(12)$, the action of ${\cal S}_4$ on $t_0$, $t_1$, $u_0$, $u_1$, $u_2$ and $f$ is given by the following equalities:
\[
\begin{array}{lll}
t_0^{\tau} = -t_0           & et & t_0^{\sigma} = -t_1,\\
t_1^{\tau} = t_1+ t_0   & et & t_1^{\sigma} = -t_0,\\
u_0^{\tau} = u_0         & et & u_0^{\sigma} = -(u_0+u_1+u_2),\\
u_1^{\tau} =u_1+u_2 & et & u_1^{\sigma} = u_0,\\
u_2^{\tau} = - u_2       & et & u_2^{\sigma} = u_1,\\
f^{\tau} = f                     & et &f^{\sigma} = f.\\
\end{array}
\]
Using the {\tt Magma} code of Appendix \ref{AppendixSymetrisation}, we get the following covariants:
\begin{align*}
   c_{0,2} = & -3a_1a_3 + a_2^2 + 12a_4a_0,\\
   c_{0,3} = &\:\: -27/2a_1^2a_4 + 9/2a_1a_2a_3 - a_2^3 + 36a_2a_4a_0 - 27/2a_3^2a_0,\\
   c_{4,1} = &\:\: a_0z^4 + a_1xz^3  + a_2x^2z^2 + a_3x^3z + a_4x^4,\\
   c_{4,2} = &\:\: (a_1^2 - 8/3a_2a_0)z^4 + (4/3a_1a_2 - 8a_3a_0)xz^3 + (4/3a_2^2 - 2a_1a_3 - 16a_4a_0)x^2z^2\\
                   & +( 4/3a_2a_3 - 8a_1a_4)x^3z + (a_3^2 - 8/3a_2a_4)x^4,\\
   c_{6,3} = &\:\: (a_1^3 - 4a_1a_0a_2+ 8a_0a_3)z^6  + (2a_1^2a_2 + 4a_0a_1a_3- 8a_0a_2^2 + 32a_0^2a_4)xz^5 + \\
                     & (5a_1^2a_3 + 40a_0a_1a_4 - 20a_0a_2a_3)x^2z^4 + (20a_1^2a_4 - 20a_0a_3^2)x^3z^3 +\\
                     & (20a_1a_2a_4- 5a_1a_3^2 - 40a_0a_3a_4)x^4z^2 +(8a_2^2a_4 - 4a_1a_3a_4  - 2a_2a_3^2  - 32a_0a_4^2)x^5z + \\
                     & (4a_2a_3a_4 - 8a_1a_4^2 - a_3^3)x^6.
\end{align*}
We recover the classic covariants of characteristic zero.

We apply the same process in characteristic $3$ (we change {\tt FF:= Rationals();} in the {\tt Magma} code by {\tt FF:= GF(3);}) and we get
\begin{align*}
   c_{0,1} = & a_2,\\
   c_{0,6} = & a_0^3a_4^3 + a_0^2a_2^2a_4^2 + a_0a_1a_2^2a_3a_4 + a_0a_2^4a_4 + 2a_0a_2^3a_3^2 + 2a_1^3a_3^3 + 2a_1^2a_2^3a_4 + a_1^2a_2^2a_3^2,\\
   c_{4,1} = &\:\: a_0z^4 + a_1xz^3  + a_2x^2z^2 + a_3x^3z + a_4x^4,\\
   c_{4,4} = & a_2c_{4,3},\\
   c_{6,3} = & (2a_0^2a_3 + 2a_0a_1a_2 + a_1^3) + (2a_0^2a_4 + a_0a_1a_3 + a_0a_2^2 + 2a_1^2a_2)x + (a_0a_1a_4 + a_0a_2a_3 + 2a_1^2a_3)x^2 + \\
                    & (a_0a_3^2 + 2a_1^2a_4)x^3 + (2a_0a_3a_4 + 2a_1a_2a_4 + a_1a_3^2)x^4 + (a_0a_4^2 + 2a_1a_3a_4 + 2a_2^2a_4 + a_2a_3^2)x^5 +\\
                    & (a_1a_4^2 + a_2a_3a_4 + 2a_3^3)x^6,\\
   c_{8,4} = & c_{4,1}(c_{4,3}-a_2^2c_{4,1}),\\
   c_{8,6} = & (c_{4,3}-a_2^2c_{4,1})c_{4,3}.
\end{align*}
Where
\begin{align*}
c_{4,3} = & (a_0a_4^2 + 2a_1a_3a_4 + 2a_2^2a_4 + a_2a_3^2)x^4 + (a_0a_3a_4 + a_1a_2a_4 + 2a_1a_3^2)x^3z + \\
                 & (a_0a_1a_4 + a_0a_2a_3 + 2a_1^2a_3)xz^3 + (a_0^2a_4 + 2a_0a_1a_3 + 2a_0a_2^2 + a_1^2a_2)z^4.
\end{align*}
$c_{4,3}$ is a covariant of degree $3$ and order $4$ in characteristic $3$. 
This system is not generator of the algebra of covariants because we cannot find $c_{4,3}$ as polynomial in the $c_{2i,j}$. 
This case provides an example where a sub-set of separating is not a generator system of the covariant algebra. 
We point out that $\{c_{0,1},c_{0,6},c_{4,1},c_{4,3},c_{6,3}\}$ is a separating system of  ${\cal C}_4$ and one wonders if it is also a generator system. 
The theorem \cite[Th.2.3.12]{DerKem} would in theory lead to an algorithm to test this hypothesis but in practice the computation did not finish.
\end{ex}

\section{A new way to generate covariants in small characteristic}
Here we introduce a new way to build covariants in small characteristic. 
To show the validity of our approach, our first idea was to use the differential characterization of covariants, as in \cite[p.43]{MR1266168}. 
It turns out however that the result of Hilbert (theorem \fullref{theoCovTroisCond}), originally shown in characteristic $0$, admits counterexamples in small characteristic, as discussed in section \ref{sec:cond}. 
So we approach the proof of the theorem \fullref{propPetitCov} directly. 
First we recall the result of Hilbert and then we give our proof. 
In the following, $f$ is a binary form defined over the field $k$ :
$$f = \sum_{i=0}^na_ix^iz^{n-i}.$$

\subsection{Hilbert's differential characterization of covariants}\label{sec:cond}
Let :
\begin{itemize}
\item $k[a_0,\dots,a_n]_d$ be the homogeneous polynomial algebra of degree $d$, 
\item $\mathbb{T}$ be the sub-group of diagonal matrices of $\SL_2(k)$,
\item $\Gamma$ be the sub-group of upper triangular matrices and diagonal equal to $1$ of $\SL_2(k)$,
\item $\Gamma^*$ be the sub-group of lower triangular matrices and diagonal equal to $1$ of $\SL_2(k)$.
\end{itemize}
These three subgroups are important because they generate $\SL_2(k)$ and thus permit to break down the issues of invariance under the action of these groups. 
Let  $M=a_0^{\rho_0}a_1^{\rho_1}\dots,a_ n^{\rho_n}$ be a monomial of $ k[a_0,\dots,a_n]$. 
We define the \textbf{weight} of $M$ by 
$w = \sum_{i=0}^ni\rho_i$.
We say that a non zero element $I$ of $ k[a_0,\dots,a_n]$ is \emph{isobaric}\index{isobaric} if all of its monomials have the same weight.
We define two differential operators on $I$ that preserve the degree.
The operators $\boldsymbol{\Delta}$ and $\textbf{D}$ are given by:
$$\boldsymbol{\Delta} = \sum_{i=1}^{n}ia_i\frac{\partial}{\partial a_{i-1}}$$
$$\textbf{D} = \sum_{i=0}^{n-1}(n-i)a_{i}\frac{\partial}{\partial a_{i+1}}.$$
\begin{theo}
Suppose $p=0$ or $p>nd +m$. 
The polynomial $C = \sum_{i=0}^mC_ix^iz^{m-i}$
is a covariant of the binary form  $f$ under the action of $\SL_2(k)$ if and only if the following conditions are satisfied:
\begin{enumerate}
\item\:$C_0,\dots,C_m$ are homogeneous functions of degree $d$ and isobaric of weight $w,w+1,\dots, w+m$ with $nd-2w = m$,
\item\:$\textbf{D} C = x\frac{\partial C}{\partial z}$,
\item\:$\boldsymbol{\Delta} C = z\frac{\partial C}{\partial x}.$
\end{enumerate}
\label{theoCovTroisCond}
\end{theo}
This result is not available in every characteristic.  
Let $f=\sum_{i=0}^{16}a_ix^iz^{16-i}$ be a binary form of degree $n = 16$ in characteristic $3$. 
Note $C= a_{11}x^6$ a homogeneous polynomial of degree $m= 6$. 
The polynomial $C$ is not a covariant of  $f$ because for $M=
\left(
\begin{array}{ll}
1 & 1\\
0 & 1
\end{array}
\right)\in\SL_2(k)$ we have
$$ C(M.f,M.(x,z))= (x + 2z)^6 (a_{11} + a_{14})\neq C.$$
Nonetheless
\begin{enumerate}
\item\:$C_0,\dots,C_m$ are homogeneous functions of degree $d = 1$ and isobaric of weight $w=5,6,\dots, 11$ with $nd-2w = 16\cdot1-2\cdot5 = 6 = m$,
\item $\textbf{D} C = 6a_{10}x^6 = 0 = x\frac{\partial C}{\partial z}$,
\item $\boldsymbol{\Delta} C = 12a_{6}x^6 = 0 = z\frac{\partial C}{\partial x}$.
\end{enumerate}

Hence, Hilbert theorem cannot be used to prove our next result. 
However, we will revise some elements of the proof using the three sub-groups $\Gamma$, $\Gamma^*$ and $\mathbb{T}$.

\subsection{A new way to build covariants in positive characteristic }

Before showing our theorem, we set some notations. 
Let $M \in \SL_2(k)$. 
We have
$$f(M.(x,z)) =\sum_{i=0}^na'_ix^iz^{n-i}.$$
In the following, we note $X=(x,z)$, $X'=M^{-1}(x,z)$, $a = (a_0,\dots,a_n)$ and $a' = (a'_0,\dots,a'_n)$. 
We start with a lemma.
\begin{lem}
An homogenous polynomial $C\in k[a_0,\dots,a_n]_d[x,z]$ is a covariant under the action of $\mathbb{T}$ if and only if
the $C_i$ are isobaric of weight $w+i$ and $nd-2w = m$.
 \label{lemCovCondTorique}
\end{lem}
\begin{proof}
Write $$C = \sum_{i=0}^mC_ix^iz^{m-i}.$$
If $M=\binom{\lambda^{-1}\: 0}{0\:\:\:\:\: \lambda}\in\mathbb{T}$, then 
$a'_i = \lambda^{n-2i}a_i$
 and $C_l(a') = \sum_{i=1}^{l}\prod_{j=0}^{n}{a'_j}^{\epsilon_{i,j,l}}$ (resp.  $C_l(a) = \sum_{i=1}^{l}\prod_{j=0}^{n}{a_j}^{\epsilon_{i,j,l}}$) with $l\in\{0,\dots,m\}$. 
We have
\begin{align*}
C_l(a') = \sum_{i=1}^{l}\prod_{j=0}^{n}\lambda^{(n-2j)\epsilon_{i,j,l}}{a_j}^{\epsilon_{i,j,l}} & = \sum_{i=1}^{l}\lambda^{\sum_{j=0}^n(n-2j)\epsilon_{i,j,l}}\prod_{j=0}^{n}{a_j}^{\epsilon_{i,j,l}}\\
& = \sum_{i=1}^{l}\lambda^{nd -2\sum_{j=0}^nj\epsilon_{i,j,l}}\prod_{j=0}^{n}{a_j}^{\epsilon_{i,j,l}}.
\end{align*}
Since $M$ also acts on $(x,z)$ by $M^{-1}.(x,z)$, we get
$$M.C(a,X) = C(a'_0,\dots,a'_n,\lambda x,\lambda^{-1} z) =  \sum_{l=0}^m\lambda^{2l -m}C_l(a')x^lz^{m-l}.$$
\textbf{Suppose that $C$ is a covariant}, then $M.C = C$, so for all $l\in\{0,\dots,m\}$
$$C(a) = \lambda^{2l-m}C_l(a').$$
This implies that for all $l$ and for all $i$
$$nd -2\sum_{j=0}^nj\epsilon_{i,j,l}+2l-m = 0.$$
In particular, $\sum_{j=0}^nj\epsilon_{i,j,l} -l$ does not depend on the $l$ or $i$.
So, we can define $w$ by putting $w =\sum_{j=0}^nj\epsilon_{i,j,l}-l$. 
We get then $nd - 2w= m$. 
Moreover, the integer $w$ is the weight of $C_0$. 
The weight of $C_l$ is $\sum_{j=0}^nj\epsilon_{i,j,l} = w+l$.\\ \\
\textbf{Conversely we want to prove that $C_l(a) = \lambda^{2l-m}C_l(a')$}. 
Since the weight of $C_l$ is $\sum_{j=0}^nj\epsilon_{i,j,l}$, we have $w =\sum_{j=0}^nj\epsilon_{i,j,l} -l$. 
Moreover, $nd-2w =m$, hence :
$$nd -2\sum_{j=0}^nj\epsilon_{i,j,l}+2l-m = 0.$$
This implies that:
$$C_l(a) = \lambda^{2l-m}C_l(a').$$
So, $C$ is a covariant under the action of $\mathbb{T}$.
\end{proof}
\noindent
Since, starting from $C_0$, we get the weight of $C_i$ of the covariant $C$, we can say that $w$ is the weight of $C$. 

\label{subsubsecCovDerMultiple}
\begin{theo}
\label{propPetitCov}
\label{theoPetitCov}
Let
${\cal Q} = \sum_{i=0}^{m_0}{\cal Q}_ix^iz^{m_0-i}$
be a covariant of $f$ of order $m_0$, degree $d_0$ and weight $\omega_0$. 
Let $l$ be an integer smaller than $m_0/2$ and $p$. 
The polynomial 
$$C = \frac{1}{z^{l}}\frac{\partial^l {\cal Q}}{\partial x^l}$$
is a covariant of $f$ if and only if $m_0-l+1$ is congruent to $0$ modulo $p$. 
When $C$ is a non zero covariant, its order is $m_0-2l$ and its degree is $d_0$.
\end{theo}
\begin{rem}
The operator was already known by Hilbert (cf \cite[th p.103]{MR1266168}). 
But the way to use it in small characteristic with the previous condition is new.
\end{rem}
\noindent
To show that $C$ is a covariant under the action of $\SL_2(k)$, we consider the action of $\mathbb{T}$, $\Gamma$ and $\Gamma^*$. 
First we analyse the action of this three sub-groups on $C$ and then we give the proof of the theorem \ref{theoPetitCov}.
Write again $$C = \sum_{i=0}^mC_ix^iz^{m-i}.$$

\paragraph{Action of $\mathbb{T}$.}
\label{parT}
By definition of $C$, $C_0,\dots,C_m$ are homogeneous functions of degree $d_0$ and isobaric of weight $l+\omega_0,l+\omega_0+1,\dots,l+\omega_0+m$. 
We express $C$ according to the coefficients of $\cal Q$
$$C = \sum_{i=l}^{m_0}\frac{i!}{l!}{\cal Q}_ix^{i-l}z^{m_0-i-l}.$$
If $p|(m_0-l+1)$, then for all $i\in\{m_0-l+1,\dots,m_0\}$,
$$p\:|\:\frac{i!}{l!}.$$
In this case, if $C$ is non zero, $C$ is a homogeneous polynomial of degree $m = m_0-2l$. 
Moreover, ${\cal Q}$ being a covariant, lemma\fullref{lemCovCondTorique} ensures that $m_0 = nd_0 - 2w_0$. 
The order of $C$ can be writen $m = nd_0 - 2(\omega_0 +l)$. 
So, by lemma \fullref{lemCovCondTorique}, $C$ is a covariant under the action of $\mathbb{T}$.
The converse is also given by lemma \fullref{lemCovCondTorique}. 
The condition $p|(m_0-l+1)$ is then a necessary and sufficient condition for $C$ to be a covariant under the action of $\mathbb{T}$. 

\paragraph{Action of $\Gamma$.}
\label{parGamma}
We set $g: (a,X)\to(a',(x + \mu z, z))$. 
We aim at showing that $C\circ g = C$. 
Meaning
$$(\frac{1}{z^{l}}\frac{\partial^l {\cal Q}}{\partial x^l})\circ g = \frac{1}{z^{l}}\frac{\partial^l {\cal Q}}{\partial x^l}.$$
This is equivalent to
$$\frac{\partial^l {\cal Q}}{\partial x^l}\circ g = \frac{\partial^l {\cal Q}}{\partial x^l}.$$
However, ${\cal Q}$ being a covariant under the action of $\Gamma$, we have 
$$\frac{\partial{\cal Q}}{\partial x} = \frac{\partial{\cal Q}\circ g}{\partial x}.$$
Moreover, 
$$\frac{\partial{\cal Q}\circ g}{\partial x} = \frac{\partial{\cal Q}}{\partial x}\circ g.$$
By immediate recurrence, we obtain the desired result.
So, $C$ is covariant under the action of $\Gamma$.

\paragraph{Action of $\Gamma^*$.}
\label{parGammaStar}
We set $g: (a,X)\to(a',(x,\mu x +z))$. 
We want to prove that $C\circ g = C$. 
Meaning
$$(\frac{1}{z^{l}}\frac{\partial^l {\cal Q}}{\partial x^l})\circ g = \frac{1}{z^{l}}\frac{\partial^l {\cal Q}}{\partial x^l}.$$
This is equivalent to
$$z^l(\frac{\partial^l {\cal Q}}{\partial x^l})\circ g = (\mu x +z)^l\frac{\partial^l {\cal Q}}{\partial x^l}.$$
Using the fact that $\displaystyle\frac{\partial^l{\cal Q}}{\partial x^l}\circ g = \sum_{i=0}^{l}\binom{l}{i}\cdot (-\mu)^{l-i} \cdot \frac{\partial^l{\cal Q}}{\partial x^i\partial^{l-i}z}$, this amounts to show
$$z^l\sum_{i=0}^{l}\binom{l}{i}\frac{\partial^l{\cal Q}}{\partial x^i\partial z^{l-i}}(-\mu)^{l-i} = \sum_{i=0}^{l}\binom{l}{i}\frac{\partial^l {\cal Q}}{\partial x^l}\mu^{l-i}x^{l-i}z^i.$$
This is still equivalent to
$$\sum_{i=0}^{l}\binom{l}{i}\mu^{l-i}\Big[\frac{\partial^l {\cal Q}}{\partial x^l}x^{l-i}z^i + (-1)^{l-i+1}\frac{\partial^l{\cal Q}}{\partial x^i\partial z^{l-i}}z^l\Big] = 0$$
or that for all $i\in\{0,\dots,l\}$
$$\frac{\partial^l {\cal Q}}{\partial x^l}x^{l-i}z^i + (-1)^{l-i+1}\frac{\partial^l{\cal Q}}{\partial x^i\partial z^{l-i}}z^l = 0.$$
Assume that $p|(m_0-l+1)$. 
We develop the left hand of the expression and we get
\begin{align*}
&0 = \sum_{j=l}^{m_0}{\cal Q}_jj(j-1)\dots(j-l+1)x^{j-i}z^{m_0+i-j} + \\
&(-1)^{l-i+1}\sum_{j=i}^{m_0-l+i}{\cal Q}_jj(j-1)\dots(j-i+1)x^{j-i}(m_0-j)(m_0-j-1)\dots(m_0-j-l+i+1)z^{m_0-j+i}.
\end{align*}
For all  $j \in\{m_0-l,\dots,m\}$, $p|j(j-1)\dots(j-l+1)$. 
In the same way, for all  $j \in\{m_0-l,\dots,m_0-l+i\}$, $p|j(j-1)\dots(j-i+1)$. 
So, the sums shall stop at $m_0-l$. 
For all $j\in\{i,\dots,l-1\}$, $p|(m_0-j)(m_0-j-1)\dots(m_0-j-l+i+1)$. 
So the two sums begin at $l$. 
Finally, since $p|(m_0-l+1)$, 
$$(m_0-j)(m_0-j-1)\dots(m_0-j-l+i+1) \equiv (l-1-j)(l-2-j)\dots(-j+i) \equiv (-1)^{l-i}(j-i)\dots(j-l+1)\pmod p$$
This proves the nullity of the expression. 
So it has been shown that if $p|(m_0-l+1)$ then $C$ is invariant under the action of $\Gamma^*$. 

\begin{proof}[Proof of theorem \fullref{propPetitCov}]
According to paragraph `Action of $\mathbb{T}$' (p. \pageref{parT}), $C$ is a covariant under the action of $\mathbb{T}$ if and only if $p|(m_0-l+1)$. 
According to paragraph`Action of  $\Gamma$' (p. \pageref{parGamma}), $C$ is a covariant under the action of $\Gamma$. 
According to paragraph `Action of  $\Gamma^*$' (p. \pageref{parGamma}), if $p|(m_0-l+1)$ then $C$ is a covariant under the action of $\Gamma^*$. 
Since $\SL_2(k)$ is generated by $\mathbb{T}$, $\Gamma$ and $\Gamma^*$, if $p|(m_0-l+1)$, $C$ is a covariant under the action of $\SL_2(k)$.

Conversely, assume that $C$ is a covariant under the action of $\SL_2(k)$. 
The invariance under the action of $\mathbb{T}$ shows that $p|(m_0-l+1)$.
\end{proof}
\label{newcov}
Thanks to this theorem, we can construct new covariants which do not appear in characteristic zero.
\begin{itemize}
\item For the binary quartic in characteristic  $3$ (cf. example \ref{exCov4}), we find $c_{0,1}$ (${\cal Q} = f$ and $l=2$) and $c_{4,3}$ (${\cal Q} = c_{6,3}$ and $l=1$) ; 
\item For the binary sextic in characteristic $3$ (cf. \cite[sec 5.2.6]{UlpatRovetta}), we find the covariant $q$ (${\cal Q} = f$ and $l=1$) of degree $1$ and order $4$ ;
\item For the binary sextic in characteristic  $5$  (cf. \cite[sec 5.2.6]{UlpatRovetta} and \cite[sec 6.6.2.3]{UlpatRovetta}), we find the covariant $c$ (${\cal Q} = f$ and $l=2$) of degree $1$ and order $2$ ;
\item For the binary octavic in characteristic $5$, we find the same invariants, $C = a_4$ (${\cal Q} = f$ and $l=4$) of degree $1$ identified by Basson and Lercier. 
\end{itemize}
It is tempting to wonder whether it is posible, in small characteristic, to get a generator system of covariant by adding this new operation. 
A first difficulty is the following. 
Let ${\cal Q}_1,\dots,{\cal Q}_r$ be covariants, $l_1,\dots,l_r$ be integers such that
$$C_i = \frac{1}{z^{l_i}}\frac{\partial^{l_i} {\cal Q}_i}{\partial x^{l_i}}$$
are the covariants obtained by this new operation starting from ${\cal Q}_i$. 
Let ${\cal Q}$ be an element of $k[{\cal Q}_1,\dots,{\cal Q}_r,C_1,\dots,C_r]$. 
The expression $\frac{1}{z^{l}}\frac{\partial^l {\cal Q}}{\partial x^l}$ is not necessarily in $k[{\cal Q}_1,\dots,{\cal Q}_r,C_1,\dots,C_r]$.
For instance, over $k=\FF_5$, for $r=1$ consider only the sextic binary form ${\cal Q}_1 = f$. 
The covariant of $f$
$$\frac{1}{z^{3}}\frac{\partial^3 f^2}{\partial x^3} = (a_3a_6 + a_4a_5)x^6 + (4a_2a_6 + 4a_3a_5 + 2a_4^2)x^5z + (a_0a_4 + a_1a_3 + 3a_2^2)xz^5 + (4a_0a_3 + 4a_1a_2)z^6$$
is not in the algebra generated by $f$ and $C_1 = \frac{1}{z^{2}}\frac{\partial^2 f}{\partial x^2}$. 
Indeed, if it was in this algebra, it would be a linear combination of $f^2$, $fC_1$ and $C_1^2$ since these are the only terms of degree $2$ in $a_i$. 
However the terms that do not depend on $x$ in these three covariants are $a_0^2$, $2a_0a_2$ and $4a_2^2$. 
We cannot generate the coefficient $(4a_0a_3 + 4a_1a_2)$.
So, it is difficult to see when the new operation will saturate the algebra. 
Actually, we even have an example where it does not. 
Consider the invariant $c_{0,6}\in{\cal I}_4$ in characteristic $3$ of the example \ref{exCov4}. 
It cannot be obtained using our new operator. 
To get it by our operation, it would have to be the derivative $l$ times starting from a certain covariant of order $m$ and degree $6$. 
The integers $m$ and $l$ have to verify $l<m/2$, $m-2l=0$ and $m-l+1$ is a multiple of $3$.
So we get this invariant by taking the second derivative of a certain covariant $c_{4,6}$ of order $4$ and  degree $6$. 
However, by performing the computations, we find that the algebra of covariant degree less than $6$ generated by our operator on the reduction of covariants of characteristic zero is generated by $c_{0,1}$, $f=c_{4,1}$, $c_{4,3}$ and $c_{6,3}$. 
The only two options for $c_{4,6}$ are $c_{0,1}^5c_{4,1}$ and $c_{0,1}^3c_{4,3}$. 
These two options do not give $c_{0,6}$.
$\:$\\ \\
\noindent
{\bf Acknowledgement.}
It is a pleasure to thank Christelle Klein Scholz for his relecture.

\bibliographystyle{alpha}
\bibliography{Bib}

\appendix
\section{Appendix}
\label{AppendixSymetrisation}
\begin{verbatim}
symmetrisation:=function(C,P4)
	P:=Parent(C);
	F:=BaseRing(P);
	r:=Rank(P);
	P2:=PolynomialRing(F,r-1);
	P3:=PolynomialRing(F,r-1);
	f:=hom<P -> P2 | [P2.i : i in [1..r-1]] cat [1]>;
	x:=P.r;
	L:=f(Coefficients(C,x));
	L2:=[];
	for s in L do
		b,t:=IsSymmetric(s,P3);
		if b then 
			L2:=L2 cat [t];
		else
			 return "not symmetric";
		end if;
	end for;
	f2:=hom< P3 -> P4 | [P4.i : i  in [1..r-1]]>;
	return &+[f2(L2[i])*(-P4.r)^(i-1) : i in [1..#L2]];
end function;


FF:= Rationals();
// FF:= GF(3);
A<x1,x2,x3,x4,x>:= PolynomialRing(FF,5);
// Order 0
t0 := (x2-x1)*(x4-x3);
t1 := (x4-x1)*(x3-x2);
//Order 2
u0 := (x-x1)*(x-x2)*(x4-x3);
u1 := (x-x1)*(x-x4)*(x3-x2);
u2 := (x-x3)*(x-x4)*(x2-x1);
//Order 4
f := (x-x1)*(x-x2)*(x-x3)*(x-x4);
M1:= Matrix(FF,[
               [0,-1,0,0,0,0],
               [-1,0,0,0,0,0],
               [0,0,-1,-1,-1,0],
               [0,0,1,0,0,0],
               [0,0,0,1,0,0],
               [0,0,0,0,0,1]
               ]);
// representation of the action of the cycle (123456)

M2:=Matrix(FF,[
               [-1,0,0,0,0,0],
               [1,1,0,0,0,0],
               [0,0,1,0,0,0],
               [0,0,0,1,1,0],
               [0,0,0,0,-1,0],
               [0,0,0,0,0,1]
               ]);
// representation of the action of the cycle (12)

GT :=  MatrixGroup<6, FF| [M1,M2]>;
// Group generated by the matrices M1 and M2

R:=InvariantRing(GT);
// Invariant ring of the group G on a set of 6 points

F:=FundamentalInvariants(R);
// Invariants who generate the ring R

L:=[Evaluate(g,[t0,t1,u0,u1,u2,f]) : g in F];
L2:=MinimalAlgebraGenerators(L);
P4<a1,a2,a3,a4,z>:=PolynomialRing(FF,5);
L3:=[symmetrisation(C,P4) : C in L2];
// L3 is the list of elements of B_{reg,sym} expressed with the coefficients of ai f

[Factorization(C): C in L3];

\end{verbatim}

\noindent\\ \\ \\ \\ \\ \\
  Florent Ulpat Rovetta\\[0.2cm]%
  Institut de Math{\'e}matiques de Marseille,\newline %
  UMR 6206 du CNRS,\newline %
  Luminy, Case 907,\newline %
  13288 Marseille\\[0.2cm]
  France\\[0.2cm]
  florent.rovetta@gmail.com
\end{document}